\documentclass[12pt, reqno]{amsart}
\usepackage[utf8]{inputenc}
\usepackage{amsmath}
\usepackage{amsfonts}
\usepackage{amssymb}
\usepackage{amsthm}
\usepackage{bm}
\usepackage{bbm}
\usepackage{xcolor}
\usepackage{url}
\usepackage{hyperref}

\numberwithin{equation}{section}

\theoremstyle{plain}
\newtheorem{theorem}{Theorem}[section]
\newtheorem{lemma}[theorem]{Lemma}
\newtheorem{corollary}[theorem]{Corollary}
\newtheorem{proposition}[theorem]{Proposition}
\newtheorem*{conjecture}{Conjecture A}
\newtheorem*{remark}{Remark}

\title[Nontrivial rational points on Erd\H{o}s-Selfridge curves]{Nontrivial rational points on Erd\H{o}s-Selfridge curves}
\author{Kyle Pratt}
\address{Brigham Young University, Department of Mathematics, Provo, UT 84602, USA}
\email{kyle.pratt@mathematics.byu.edu}

\subjclass[2010]{}

\allowdisplaybreaks

\begin{document}
\date{}

\maketitle

\begin{abstract}
We study rational points on the Erd\H{o}s-Selfridge curves 
\begin{align*}
y^\ell = x(x+1)\cdots (x+k-1),
\end{align*}
where $k,\ell\geq 2$ are integers. These curves contain ``trivial'' rational points $(x,y)$ with $y=0$, and a conjecture of Sander predicts for which pairs $(k,\ell)$ the curve contains ``nontrivial'' rational points where $y\neq 0$. 

Suppose $\ell \geq 5$ is a prime. We prove that if $k$ is sufficiently large and coprime to $\ell$, then the corresponding Erd\H{o}s-Selfridge curve contains only trivial rational points. This proves many cases of Sander's conjecture that were previously unknown. The proof relies on combinatorial ideas going back to Erd\H{o}s, as well as a novel ``mass increment argument'' that is loosely inspired by increment arguments in additive combinatorics. The mass increment argument uses as its main arithmetic input a quantitative version of Faltings's theorem on rational points on curves of genus at least two.
\end{abstract}

\section{Introduction}

There is a venerable tradition in number theory of finding integer or rational points on curves (see e.g. \cite{Dic1966}). Many interesting arithmetic questions can be stated naturally in terms of points on curves. Consider, for instance, the well-known parameterization of Pythagorean triples in terms of rational points on the unit circle $x^2+y^2=1$, or the relationship between Fermat's Last Theorem and rational points on the Fermat curves $x^\ell+y^\ell = 1$.

Liouville \cite{Lio1857} asked in 1857 whether the product of consecutive positive integers is ever a perfect power. This is equivalent to determining whether there are positive integers $x,y$ on the curve
\begin{align}\label{eq:ES curve}
y^\ell = x(x+1) \cdots (x+k-1), \ \ \ \ \ \ k,\ell \geq 2.
\end{align}
The curve \eqref{eq:ES curve} clearly contains ``trivial'' integral points where $y=0$. More than one hundred years after Liouville posed his question, Erd\H{o}s and Selfridge \cite{ErdSel1975} solved the problem by showing \eqref{eq:ES curve} contains only these trivial integral points. In contrast to previous work \cite{Erd1939,Rig1940} on the problem, which had utilized deep techniques, the solution of Erd\H{o}s and Selfridge is essentially elementary (see also \cite{Erd1955}).

Beyond the study of \emph{integral} points, it is natural to require a complete classification of the \emph{rational} points on the Erd\H{o}s-Selfridge curve \eqref{eq:ES curve}. Note that \eqref{eq:ES curve} is an affine model of a smooth projective curve which has genus $1 + \frac{1}{2}\left(k\ell - \ell - k - \text{gcd}(\ell,k) \right)$ (see \cite[p. 87, Exercise A.4.6]{HinSil2000}). It follows the genus is $\geq 2$ if $k+\ell \geq 7$. By Faltings's celebrated theorem \cite{Fal1983}, the curve \eqref{eq:ES curve} then has only finitely many rational points. One may therefore list, at least in principle, the finitely many rational points on the Erd\H{o}s-Selfridge curve when $k+\ell \geq 7$.

As noted above, the Erd\H{o}s-Selfridge curve contains the trivial rational points $(x,y) = (-n,0)$ where $0\leq n \leq k-1$. It is natural to wonder whether there are ``nontrivial'' rational points on the curve, i.e. points with $y\neq 0$. Sander \cite{San1999} studied the curves \eqref{eq:ES curve} for $2\leq k\leq 4$, and on the basis of his investigations conjectured that nontrivial rational points exist only in the case $(k,\ell)=(2,2)$. Bennett and Siksek \cite{BenSik2016}, relying in part on computations in \cite{BBGH2006}, noted the existence of some additional nontrivial rational points and stated a corrected version of Sander's conjecture.

\begin{conjecture}[Sander]
If $k,\ell\geq 2$ are integers, then the only rational points $(x,y) \in \mathbb{Q}^2$ on the Erd\H{o}s-Selfridge curve \eqref{eq:ES curve} satisfy $y=0$, or
\begin{itemize}
\item $(x,y,k,\ell) = \left(a^2/(b^2-a^2),ab/(b^2-a^2),2,2 \right)$, for integers $a \neq \pm b$,
\item $(x,y,k,\ell) = \left((1-2j)/2, \pm 2^{-j} \prod_{i=1}^j(2i-1),2j,2 \right)$, for $j\geq 2$ an even integer,
\item $(x,y,k,\ell) = \left(-4/3, 2/3,3,3 \right)$,
\item $(x,y,k,\ell) = \left(-2/3,-2/3,3,3 \right)$.
\end{itemize}
\end{conjecture}
\begin{remark}
In the case $(k,\ell) = (2j,2)$, the conjecture as stated in \cite{BenSik2016} does not include the condition that $j\geq 2$ is even. This condition is necessary, however, since if $j\geq 2$ is odd then the points $(x,y) = \left((1-2j)/2,\pm 2^{-j}\prod_{i=1}^j(2i-1) \right)$ lie on the curve $-y^2 = x(x+1) \cdots (x+2j-1)$ rather than on \eqref{eq:ES curve}.
\end{remark}

Bennett and Siksek \cite[Theorem 1]{BenSik2016} utilized the modular method (see e.g. \cite{Sik2012}) and proved that if $k\geq 2$ and if $\ell$ is a prime such that \eqref{eq:ES curve} has nontrivial rational points, then $\ell < \exp(3^k)$. Obtaining an upper bound on $\ell$ in terms of $k$ is significant, since work of Darmon and Granville \cite{DarGra1995} then implies that \eqref{eq:ES curve} has at most finitely many solutions for every fixed $k\geq 3$. Das, Laishram, and Saradha \cite{DLS2018}, Edis \cite{Edis2019}, and Saradha \cite{Sar2019} proved similar results when the right-hand side of \eqref{eq:ES curve} has some terms $x+i$ deleted. Das, Laishram, Saradha, and Sharma \cite{DLSS2023} studied these variants for small $k$. Saradha \cite{Sar2021} also studied the size of the denominator of a nontrivial rational point on \eqref{eq:ES curve}.

The study of rational points on Erd\H{o}s-Selfridge curves is related to a more general Diophantine problem. We observe that the right-hand side of \eqref{eq:ES curve} is the product of $k$ consecutive terms in an arithmetic progression with common difference equal to one. One may consider more general arithmetic progressions and ask whether there are any solutions to the equation
\begin{align}\label{eq:generalized Erdos conj}
y^\ell = n(n+d)(n+2d) \cdots (n+(k-1)d), \ \ \ \ \ \ \text{gcd}(n,d)=1,
\end{align}
in, say, positive integers $n,k,d,y,\ell$ with $k,\ell\geq 2$. (We require $\text{gcd}(n,d)=1$ in order to rule out artificial solutions.) There is a conjecture, usually ascribed to Erd\H{o}s \cite[p. 356]{BenSik2020}, that \eqref{eq:generalized Erdos conj} has no solutions if $k\geq k_0$ is sufficiently large. Note that \eqref{eq:generalized Erdos conj} has infinitely many solutions when $(k,\ell) = (3,2)$.

A number of authors have worked on \eqref{eq:generalized Erdos conj} and subsequent generalizations (see the survey by Shorey \cite{Sho2006} for an introduction to the extensive literature). In particular, the conjecture is known to hold when additional restrictions are imposed. For instance, the conjecture holds if $\ell$ and the number of prime factors of $d$ are fixed \cite[Corollary 3]{ShoTij1990}, and if $k$ is small and fixed \cite{BBGH2006,HLST2007,GHP2009,HTT2009,HK2011}. For somewhat larger $k$, it is known that \eqref{eq:generalized Erdos conj} has at most finitely many solutions \cite{Ben2018}. In important recent work, Bennett and Siksek \cite{BenSik2020} used the modular method and many ingenious arguments to prove that if \eqref{eq:generalized Erdos conj} does have a solution with prime exponent $\ell$ and $k$ sufficiently large, then $\ell < \exp(10^k)$.

By writing rational numbers as quotients of integers and rearranging, the study of rational points on \eqref{eq:ES curve} reduces to studying equations such as \eqref{eq:generalized Erdos conj}. In the case where $\ell$ and $k$ are coprime, one reduces to a Diophantine equation like \eqref{eq:generalized Erdos conj} where the modulus of the arithmetic progression is an $\ell$-th power. We heavily exploit this structure in the present work.

All present results covering \eqref{eq:ES curve} and \eqref{eq:generalized Erdos conj} allow the unfortunate possibility that, for a fixed integer $\ell \geq 5$, say, the curve \eqref{eq:ES curve} has nontrivial rational points for every large $k$. We view this as a ``horizontal'' regime of Sander's conjecture. This contrasts with the ``vertical'' regime studied by Bennett and Siksek \cite{BenSik2016}, in which, for a given $k$, they rule out extremely large values of $\ell$. We partially remedy this defect in the present paper, and prove Sander's conjecture in many cases when $k$ is large compared to $\ell$.

\begin{theorem}\label{thm:main theorem}
Let $k$ be a sufficiently large positive integer, and let $\ell$ be a prime with
\begin{align*}
5\leq \ell \leq (\log \log k)^{1/5}.
\end{align*}
If $\textup{gcd}(k,\ell)=1$, then any rational point $(x,y)\in \mathbb{Q}^2$ on the Erd\H{o}s-Selfridge curve
\begin{align*}
y^\ell = x(x+1)\cdots (x+k-1)
\end{align*}
has $y=0$.
\end{theorem}

Theorem \ref{thm:main theorem} and the aforementioned work of Bennett and Siksek \cite{BenSik2016} immediately yield the following corollary.

\begin{corollary}
Let $k$ be a sufficiently large positive integer, and assume Conjecture A is false for some prime $\ell \geq 2$ with $\textup{gcd}(k,\ell)=1$. Then $\ell \leq 3$, or
\begin{align*}
(\log \log k)^{1/5} < \ell < \exp(3^k).
\end{align*}
\end{corollary}


The hypothesis $\text{gcd}(k,\ell)=1$ in Theorem \ref{thm:main theorem} is critical for our work. It is crucial for our method that the common difference of the arithmetic progression in \eqref{eq:generalized Erdos conj} is an $\ell$-th power, and we are only able to guarantee this condition in the event that $\ell \nmid k$.

The assumption that $\ell$ is a prime is not vital, and is made only for convenience. It would be possible, for instance, to prove a version of Theorem \ref{thm:main theorem} with $\ell=4$.

Our method would extend to studying rational points on variants of Erd\H{o}s-Selfridge curves where some terms $x+i$ are deleted. We leave the details to the interested reader.

Theorem \ref{thm:main theorem} does not apply to the cases $\ell=2$ or $3$. This is because our method requires auxiliary curves of the form
\begin{align*}
AX^\ell + BY^\ell = C, \ \ \ \ \ \ A,B, C \in \mathbb{Z}\backslash \{0\},
\end{align*}
to have genus at least two. This holds if $\ell \geq 5$, but not if $\ell \leq 3$. However, we are able to obtain a partial result for $\ell=3$. While we cannot rule out the existence of nontrivial rational points, we can show that any such point must have large height. Observe that if $3 \nmid k$, then any nontrivial rational point on the curve 
\begin{align*}
y^3 = x(x+1)\cdots (x+k-1)
\end{align*}
necessarily has the form $(x,y) = (a/d^3,b/d^k)$ with $a,b \in \mathbb{Z}\backslash\{0\}$, $d\geq 1$ a positive integer, and $\textup{gcd}(a,d)=\textup{gcd}(b,d)=1$.

\begin{theorem}\label{thm:ell = 3}
Let $k$ be a sufficiently large positive integer with $3 \nmid k$. Assume there is a nontrivial rational point $(x,y)=(a/d^3,b/d^k)$ on the curve $y^3 = x(x+1) \cdots (x+k-1)$. Then there is an absolute constant $c>0$ such that $d \geq \exp \left(k^{c/\log \log k} \right)$.
\end{theorem}

Theorem \ref{thm:ell = 3} improves upon an earlier result of Shorey and Tijdeman \cite[Corollary 4]{ShoTij1990} which implies $d\geq k^{c\log \log k}$.

The structure of the rest of the paper is as follows. We fix notation and conventions in section \ref{sec:notation}. In section \ref{sec:outline}, we sketch the proofs of Theorems \ref{thm:main theorem} and \ref{thm:ell = 3}. We make important preliminary steps in section \ref{sec:key props} and state the four key propositions on which Theorem \ref{thm:main theorem} depends. Section \ref{sec:proof of some of key props} contains the proofs for three of the four key propositions, those whose proofs are somewhat straightforward. The last of the key propositions is proved in section \ref{sec:mass increment}. In the final section of the paper, section \ref{sec:ell = 3}, we prove Theorem \ref{thm:ell = 3} by making appropriate modifications and proving some new supplementary results.

\section{Notation}\label{sec:notation}

In what follows, $\ell$ always denotes a prime with $\ell \geq 3$, and usually $\ell \geq 5$. The integer $k$ is positive and sufficiently large. Given two integers $m$ and $n$, not both zero, we write $\text{gcd}(m,n)$ for the greatest common divisor of $m$ and $n$, which is the largest positive integer which divides both of them. We write $m \mid n$ if $m$ divides $n$, and $m \nmid n$ if $m$ does not divide $n$. If $m$ and $n$ are coprime integers, then we write $\overline{m}\pmod{n}$ for the residue class modulo $n$ such that $m \overline{m} \equiv 1 \pmod{n}$.

Given a prime number $p$ and a nonzero integer $n$, we may write $n = p^r m$, where $\text{gcd}(m,p)=1$, and we write $v_p(n) = r$. We define $\omega(n)$ to be the number of distinct prime factors of $n$. The number of primes $\leq x$ is $\pi(x)$.

We say that $f \ll g$ or $f = O(g)$ if there is an absolute constant $C>0$ such that $f\leq Cg$. We write $f \asymp g$ if $f \ll g$ and $g \ll f$. Occasionally we write $O_A(\cdot)$, say, if the implied constant depends on some other quantity $A$. We write $o(1)$ for a quantity which tends to zero as some other parameter, usually $k$, tends to infinity.

We write $|S|$ or $\#S$ for the cardinality of a finite set $S$.

Given a rational number $a/q$ with $\text{gcd}(a,q)=1$ and $q\geq 1$, we write $H(a/q) = \max(|a|,|q|)$ for the naive height of $a/q$. If $P = (x,y)$ is a point on an elliptic curve $E/\mathbb{Q}$, we write $H(P) = H(x)$ for the naive height of $P$, $h(P) = \log H(P)$ for the logarithmic or Weil height of $P$, and $\hat{h}(P)$ for the canonical height of $P$. We refer the reader to a standard source \cite[section VIII.9]{Sil2009} for additional details.

\section{Outline of the proofs}\label{sec:outline}

Theorems \ref{thm:main theorem} and \ref{thm:ell = 3} have many ingredients in common. We first sketch the proof of Theorem \ref{thm:main theorem}, and then describe some of the necessary modifications for Theorem \ref{thm:ell = 3}.

We begin the proof of Theorem \ref{thm:main theorem} by assuming for contradiction that \eqref{eq:ES curve} has a rational point with $y\neq 0$. We assume $\ell \geq 5$, and that $k$ is large and coprime to $\ell$. Since $\text{gcd}(k,\ell)=1$, this implies (Lemma \ref{lem:rat point on ES implies Dioph eq with ints}) there is a solution to
\begin{align*}
n(n+d^\ell)(n+2d^\ell)\cdots (n+(k-1)d^\ell) = t^\ell
\end{align*}
in nonzero integers $n,d,t$ with $\text{gcd}(n,d)=1$ and $d\geq 1$. This, in turn, gives rise to factorizations
\begin{align}\label{eq:factorization of n plus i d ell}
n+id^\ell = a_i t_i^\ell,
\end{align}
where the $a_i$ are positive integers with all their prime factors $< k$, and $t_i$ is divisible only by primes $\geq k$ (Lemma \ref{lem:factorization of n plus i d ell}). The central focus of our method is to study the set of integers $a_i$. In particular, we prove that the set $\{a_i : a_i < k\}$ has positive density in the interval $[1,k-1]$.

We rely on the simple identity
\begin{align}\label{eq:ternary equation}
a_it_i^\ell - a_jt_j^\ell=(n+id^\ell) - (n+jd^\ell)= (i-j)d^\ell.
\end{align}
The equality of the left- and right-hand sides yields a generalized Fermat equation of signature $(\ell,\ell,\ell)$. In contrast to many previous works, however, we do not rely on modular approaches. Rather, our method relies on combinatorial arguments, and a version of Faltings's theorem provides the necessary arithmetic input.

We have little initial control over the size of the integers $a_i$, but by an elementary argument of Erd\H{o}s (Lemma \ref{lem:subset divides factorial}) we can find a set of indices $i \in I\subseteq \{0,1,\ldots,k-1\}$ with $|I| =(1-o(1))k$ such that
\begin{align}\label{eq:prod of ai divides factorial}
\prod_{i \in I} a_i \leq k!.
\end{align}
This provides some weak but sufficient control on the $a_i$. In particular, since \eqref{eq:ternary equation} implies the integers $a_i \geq k$ are distinct, an argument with \eqref{eq:prod of ai divides factorial} and Stirling's formula implies there must be $\gg k/\log k$ indices $i \in I$ such that $1\leq a_i < k$ (Lemma \ref{lem:spark for mass inc, many i with small ai}). If many of these $a_i < k$ were distinct from one another, then it might be possible to carry out an iterative procedure. To illustrate, note that from \eqref{eq:prod of ai divides factorial} we have
\begin{align*}
\prod_{\substack{i \in I \\ a_i < k}} a_i \cdot \prod_{\substack{i \in I \\ a_i \geq k}} a_i \leq k!,
\end{align*}
so if many $a_i < k$ are distinct then the first product is large, and this implies an upper bound on the second product which is smaller (i.e. better) than $k!$. This implies the existence of more indices $i \in I$ with $a_i < k$ than we had before. If, again, many of these $a_i < k$ were distinct, we could repeat the process and obtain yet another improvement. This is the basic idea of the \emph{mass increment argument}, which was loosely inspired by various increment arguments in additive combinatorics (see \cite[\S 1.10]{TaoEoR2010}, \cite[Chapters 10 and 11]{TaoVu2006}). 

The quantity of interest in the mass increment argument is the cardinality or ``mass'' of the set $\{a_i : a_i < k\}$. The main obstacle in carrying out this mass increment argument is the possibility that there are many collisions $a_i = a_j$ among the $a_i,a_j < k$.

If many of the $a_i < k$ are equal to each other, then by a pigeonhole argument (Lemma \ref{lem:not enough distinct ai gives many points on curve}) there is some auxiliary curve $X^\ell + Y^\ell = A$ which contains many rational points. The nonzero integer $A$ has size $|A| \ll (\log k)^{O(1)}$, and the curve must contain $\gg k^{1-o(1)}$ rational points. Thus, a low-height curve of genus $\geq 2$ contains a large number of rational points; we rule out this scenario by using a quantitative version of Faltings's theorem (Proposition \ref{prop:quant faltings theorem}).

The mass increment argument shows there are $\gg k$ distinct integers $a_i < k$ (Proposition \ref{prop:ai < k have pos density}). Hence, the set $\{a_i : a_i < k\}$ has positive density in $[1,k-1]$. The important consequence of this structure is that sets of integers of positive density contain many pairs having large greatest common divisor. More quantitatively, by an argument of Erd\H{o}s (Proposition \ref{prop:dense sequence has large gcds}), we can find $\gg k$ pairs $a_i\neq a_j$ with $\text{gcd}(a_i,a_j) \gg k$. Applying \eqref{eq:ternary equation} and another pigeonholing argument then implies that some curve of the form $AX^\ell + BY^\ell = C$ contains $\gg k^{1-o(1)}$ rational points, where $A,B,C$ are bounded, nonzero integers (the size of $|A|,|B|,|C|$ is bounded by an absolute constant independent of $k$ and $\ell$). This contradicts Faltings's theorem when $k$ is sufficiently large.

The rational points $(X,Y)$ which lie on the auxiliary curves $X^\ell+Y^\ell = A$ or $AX^\ell + BY^\ell = C$ are essentially of the form $(X,Y) = (t_i/d,t_j/d)$ (recall \eqref{eq:factorization of n plus i d ell}). To show these generate many distinct rational points, we need to bound the number of $i$ with $|t_i|=1$. This can be accomplished with elementary arguments (Proposition \ref{prop:not many ti are equal to one}).

For the proof of Theorem \ref{thm:ell = 3}, we assume for contradiction that there is a nontrivial rational point $(x,y)=(a/d^3,b/d^k)$ on the curve \eqref{eq:ES curve} with $\ell=3$ and $d\leq \exp (k^{\delta/\log \log k})$; here $\delta>0$ is some sufficiently small absolute constant. We still use a mass increment argument (Proposition \ref{prop:mass inc for ell = 3}), but now the crux of the proof hinges on a curve $X^3+Y^3=A$ having $\gg k^{1-o(1)}$ rational points of (naive) height $\ll (dk)^{O(1)}$. The nonzero integer $A$ has size $\ll (\log k)^{O(1)}$. At this point, an elementary argument based on the divisor bound recovers only the result of Shorey and Tijdeman mentioned in the introduction.

We exploit the arithmetic of elliptic curves to do better. The main input is a uniform upper bound on the number of rational points with bounded canonical height on an elliptic curve (Proposition \ref{prop:count points in MW lattice}). The properties of the canonical height imply that rational points cannot cluster too much in the Mordell-Weil lattice, so if an elliptic curve of small rank has $\gg k^{1-o(1)}$ rational points with naive height $\ll (dk)^{O(1)}$, then $d$ must be large. The small size of the integer $A$ and the particular form of the elliptic curve which arises allow us to obtain suitable control on the rank.

\section{Initial steps, and the key propositions}\label{sec:key props}

The following standard lemma transforms the problem of studying nontrivial rational points on the Erd\H{o}s-Selfridge curves into a related Diophantine problem in integers.

\begin{lemma}[Diophantine equation from nontrivial points]\label{lem:rat point on ES implies Dioph eq with ints}
Let $\ell \geq 3$ be a prime, and let $k \geq 2$ be a positive integer with $\text{gcd}(k,\ell)=1$. Assume there is a rational point $(x,y)$ on \eqref{eq:ES curve} with $y\neq 0$. Then there exist nonzero integers $n$ and $t$ and a positive integer $d$ with $\textup{gcd}(n,d)=1$ such that
\begin{align}\label{eq:prod over AP is ellth power}
\prod_{i=0}^{k-1} (n+id^\ell) = t^\ell.
\end{align}
\end{lemma}
\begin{proof}
Let $(x,y) = (n/b,t/u)$ be a rational point on \eqref{eq:ES curve} with $t/u \neq 0$. We may assume $nt \neq 0$, $\text{gcd}(n,b) = \text{gcd}(t,u)=1$, and $b,u\geq 1$. We insert these expressions for $x$ and $y$ into \eqref{eq:ES curve} and then multiply through by $b^ku^\ell$ to obtain
\begin{align}\label{eq:initial rat sub in ES curve}
u^\ell \prod_{i=0}^{k-1} (n+ib) = t^\ell b^k.
\end{align}
Since $u$ and $t$ are coprime we see that $u^\ell \mid b^k$, and, similarly, since $n$ and $b$ are coprime we find $b^k \mid u^\ell$. Since $b$ and $u$ are positive we deduce that $b^k = u^\ell$. The integers $k$ and $\ell$ are coprime, so there exists a positive integer $d$ such that $b = d^\ell$ and $u = d^k$. We return to \eqref{eq:initial rat sub in ES curve} to find $n(n+d^\ell) \cdots (n+(k-1)d^\ell) = t^\ell$. We have $\text{gcd}(n,d)=1$ since $b = d^\ell$ and $\text{gcd}(n,b)=1$.
\end{proof}

The next lemma is also standard (e.g. \cite[Lemma 3.1]{BenSik2020}), but critical in all that follows. It details a convenient factorization of the terms appearing in \eqref{eq:prod over AP is ellth power}.

\begin{lemma}[Factorization of terms]\label{lem:factorization of n plus i d ell}
Assume there are nonzero integers $n$ and $t$ and a positive integer $d$ with $\textup{gcd}(n,d)=1$ such that \eqref{eq:prod over AP is ellth power} holds. Then for $0\leq i \leq k-1$ there are positive integers $a_i$ and nonzero integers $t_i$ such that the following hold:
\begin{itemize}
\item $n+id^\ell$ may be factored uniquely as $n+id^\ell=a_it_i^\ell$.
\item If $p$ is a prime with $p \mid a_i$, then $p < k$.
\item If $p$ is a prime with $p \mid t_i$, then $p\geq k$.
\item $\textup{gcd}(a_i,a_j) \mid (j-i)$ for all $0 \leq i < j \leq k-1$.
\item $\textup{gcd}(a_i,d)=1$.
\item $\textup{gcd}(t_i,t_j)=1$ for all $0 \leq i < j \leq k-1$.
\item The product $\prod_{i=0}^{k-1} a_i$ is an $\ell$-th power.
\end{itemize}
\end{lemma}
\begin{proof}
Note that $n+id^\ell \neq 0$, since $t\neq 0$. Let $0 \leq i < j \leq k-1$ be two integers. Then $\text{gcd}(n+id^\ell,n+jd^\ell) \mid (j-i) d^\ell$, and since $\text{gcd}(n,d)=1$ it follows that $\text{gcd}(n+id^\ell,n+jd^\ell) \mid (j-i)$. Observe that $j-i \leq k-1$. 

We factor $n+id^\ell = a_iz_i$, where $a_i$ is divisible by primes $<k$ and $z_i$ is divisible by primes $\geq k$. Since $\text{gcd}(z_i,z_j) =1$ if $i \neq j$, it follows from \eqref{eq:prod over AP is ellth power} that $z_i = t_i^\ell$. Therefore, $\text{gcd}(t_i,t_j) = 1$ and $\text{gcd}(a_i,a_j)\mid (j-i)$. If we require that $a_i$ is positive (as we may since $\ell$ is odd), then the factorization $n+id^\ell = a_it_i^\ell$ is unique. Since $n+id^\ell = a_i t_i^\ell$ and $\text{gcd}(n,d)=1$, we obtain $\text{gcd}(a_i,d)=1$.

Lastly, since
\begin{align*}
t^\ell = \prod_{i=0}^{k-1} (n+id^\ell) = \prod_{i=0}^{k-1}a_i t_i^\ell = \prod_{i=0}^{k-1}a_i \cdot \left(\prod_{i=0}^{k-1}t_i \right)^\ell
\end{align*}
the product of the $a_i$ must be an $\ell$-th power.
\end{proof}

With Lemma \ref{lem:factorization of n plus i d ell} in place, we state the four key propositions upon which Theorem \ref{thm:main theorem} relies.

The first proposition is essentially due to Erd\H{o}s \cite[Lemma 1]{Erd1939}, and states that a set of integers of positive density contains many pairs which have a large greatest common divisor.

\begin{proposition}[Many pairs with large GCD]\label{prop:dense sequence has large gcds}
Let $k$ be a positive integer, and let $1 \leq b_1 < \cdots < b_r \leq k$ be positive integers. Let $c \in (0,1), \eta \in (0,1), A \geq 1$ be real constants such that 
\begin{align*}
\eta(A+1) + \prod_{p\leq A} \left(1 - \frac{1}{p}\right) \leq \frac{c}{2},
\end{align*}
where the product is over primes $p\leq A$. If $r \geq c k$ and $k\geq k_0(c,\eta,A)$, then there are $\geq \frac{c}{3}k$ distinct ordered pairs of integers $(b_i,b_j)$ with $b_i\neq b_j$ such that $\textup{gcd}(b_i,b_j) > \eta k$.
\end{proposition}

The second proposition states that if $k$ is sufficiently large, then there are few indices $i \in [0,k-1]$ such that $a_i < k$ and $t_i =\pm 1$.

\begin{proposition}[Few $t_i$ are trivial]\label{prop:not many ti are equal to one}
Let $\ell \geq 3$ be a prime, and let $k\geq 21$ be a positive integer with $\textup{gcd}(k,\ell)=1$. Assume there are nonzero integers $n$ and $t$ and a positive integer $d$ with $\textup{gcd}(n,d)=1$ such that \eqref{eq:prod over AP is ellth power} holds. Let the integers $a_i,t_i$ be as in Lemma \ref{lem:factorization of n plus i d ell}. Then there are at most 20 integers $i \in [0,k-1]$ such that $a_i < k$ and $|t_i| = 1$.
\end{proposition}

The third proposition states that there are a positive proportion of distinct integers $a_i$ in the interval $[1,k-1]$, where the integers $a_i$ are as in Lemma \ref{lem:factorization of n plus i d ell}. This result has the most involved proof of the four key propositions. In particular, the proof of the proposition relies on the mass increment argument sketched in section \ref{sec:outline}.

\begin{proposition}[Positive density from mass increment]\label{prop:ai < k have pos density}
Assume that $k$ is sufficiently large and $\ell$ is a prime with $5\leq \ell \leq (\log \log k)^{1/5}$ and $\textup{gcd}(k,\ell)=1$. Assume there are nonzero integers $n$ and $t$ and a positive integer $d$ with $\textup{gcd}(n,d)=1$ such that \eqref{eq:prod over AP is ellth power} holds, and let the integers $a_i$ be as in Lemma \ref{lem:factorization of n plus i d ell}. Then $\#\{a_i : 1 \leq a_i < k\} \geq 0.23 k$.
\end{proposition}

The last key proposition is a quantitative version of Faltings's theorem for certain curves; it is a special case of a theorem of R\'emond \cite{Rem2010}.

\begin{proposition}[Quantitative Faltings's theorem]\label{prop:quant faltings theorem}
Let $\ell \geq 5$ be a prime. Let $A,B,C$ be nonzero integers, and set $H = \max(|A|,|B|,|C|)$. Let $\mathcal{C}$ be the smooth projective curve with affine model $AX^\ell + BY^\ell = C$. Then the number $\#\mathcal{C}(\mathbb{Q})$ of $\mathbb{Q}$-rational points on $\mathcal{C}$ satisfies
\begin{align*}
\#\mathcal{C}(\mathbb{Q}) \leq \exp(5^{\ell^4}(\log 3H)(\log \log 3H)).
\end{align*}
\end{proposition}

\begin{proof}[Proof of Theorem \ref{thm:main theorem} assuming Propositions \ref{prop:dense sequence has large gcds}, \ref{prop:not many ti are equal to one}, \ref{prop:ai < k have pos density}, and \ref{prop:quant faltings theorem}]
Let $k$ be large, and assume $\ell$ is a prime satisfying $5\leq \ell \leq (\log \log k)^{1/5}$ with $\text{gcd}(k,\ell)=1$. Assume for contradiction that \eqref{eq:ES curve} contains a rational point $(x,y)$ with $y\neq 0$. Lemmas \ref{lem:rat point on ES implies Dioph eq with ints} and \ref{lem:factorization of n plus i d ell} then provide us with factorizations $n+id^\ell = a_it_i^\ell$ for $0\leq i \leq k-1$. By Proposition \ref{prop:ai < k have pos density} we have $\#\{a_i : 1 \leq a_i < k\} \geq 0.23 k$. According to Proposition \ref{prop:not many ti are equal to one} there are at most 20 integers $i \in [0,k-1]$ such that $a_i < k$ and $|t_i|=1$, and therefore
\begin{align*}
\#\{a_i : 1 \leq a_i < k, |t_i| \neq 1\} \geq 0.23 k - 20 \geq 0.229 k.
\end{align*}

We apply Proposition \ref{prop:dense sequence has large gcds} with $c=0.229, A=283, \eta = 1/17000$, noting that
\begin{align*}
\frac{284}{17000}+\prod_{p\leq 283}\left(1 - \frac{1}{p}\right) < 0.114499 < \frac{c}{2} = 0.1145.
\end{align*}
This yields $\gg k$ distinct ordered pairs $(a_i,a_j)$ of distinct integers $a_i,a_j<k$ such that $\text{gcd}(a_i,a_j) > k/17000$. For each such pair we have $a_it_i^\ell - a_jt_j^\ell = (i-j)d^\ell$, and factoring out the greatest common divisor of $a_i$ and $a_j$ gives $f_it_i^\ell - g_jt_j^\ell = \pm h_{i,j}d^\ell$, where $1\leq f_i,g_j,h_{i,j}\leq 17000$ are positive integers and the $\pm$ sign is taken according to whether $i-j$ is positive or negative. Now observe that the point $(t_i/d,t_j/d)$ lies on the curve $f_i X^\ell - g_j Y^\ell = \pm h_{i,j}$. Since $|t_i|,|t_j|\neq 1$ and the $t_i$'s are pairwise coprime, all of the rational points $(t_{i}/d,t_{j}/d)$ are distinct as we range over the $\gg k$ pairs $(a_i,a_j)$ provided by Proposition \ref{prop:dense sequence has large gcds}. Therefore, by the pigeonhole principle, there exist nonzero integers $|A|,|B|,|C| \leq 17000$ such that the curve with affine model $AX^\ell + BY^\ell = C$ contains $\gg k$ rational points. This contradicts Proposition \ref{prop:quant faltings theorem} since $5^{\ell^4} \leq \exp((\log 5)(\log \log k)^{4/5}) \leq \sqrt{\log k}$, say.
\end{proof}

\section{Proofs of Propositions \ref{prop:dense sequence has large gcds}, \ref{prop:not many ti are equal to one}, and \ref{prop:quant faltings theorem}}\label{sec:proof of some of key props}

In this section we give proofs for three of the four key propositions.

\begin{proof}[Proof of Proposition \ref{prop:dense sequence has large gcds}]
The proof is essentially that of \cite[Lemma 1]{Erd1939}, but we supply more details and proceed in a more quantitative fashion.

Let $\eta k < d_1 < \cdots < d_s \leq k$ be all those integers in $(\eta k , k]$ such that every proper divisor of $d_i$ is $\leq \eta k$ (i.e. every divisor $d$ of $d_i$ with $d < d_i$ satisfies $d\leq \eta k$). Write $\mathcal{D}$ for this set of $d_i$, and observe that $\#\mathcal{D} = s$.

We claim that every integer $n \in (\eta k ,k]$ is divisible by some element of $\mathcal{D}$. If $n \in \mathcal{D}$ then this is obvious, so assume $n \not \in \mathcal{D}$. Then $n$ has a proper divisor $d > \eta k$, and without loss of generality we may assume $d$ is the least such divisor of $n$. By the minimality of $d$ every proper divisor of $d$ is $\leq \eta k$, so $d \in \mathcal{D}$. This proves the claim.

There are $\geq r - \eta k$ integers $b_i$ such that $b_i \in (\eta k , k]$. By the claim, every such $b_i$ is divisible by some $d \in \mathcal{D}$. We may then define a function $f : \{b_i \in (\eta k ,k]\} \rightarrow \mathcal{D}$ given by $f(b_i) = d$, say, where $d$ is the largest element of $\mathcal{D}$ which divides $b_j$. There are at most $s$ integers $b_i$ such that $f(b_i)$ is unique, so there are $\geq r - \eta k - s$ integers $b_i$ such that $f(b_i) = f(b_j)$ for some $j \neq i$. For each such $b_i$ we arbitrarily choose some $b_j \neq b_i$ such that $f(b_j)=f(b_i)$, and this gives rise to a pair $(b_,b_j)$ such that $\text{gcd}(b_i,b_j) > \eta k$. The pairs $(b_i,b_j)$ and $(b_i',b_j')$ are distinct if $b_i \neq b_i'$, so there are $\geq r-\eta k -s$ distinct (ordered) pairs of distinct integers $b_i \neq b_j$ such that $\text{gcd}(b_i,b_j) > \eta k$.

Since $r\geq ck$ it remains to show $\eta k +s \leq \frac{2c}{3}k$. Recall $s = \#\mathcal{D}$, and set $\delta = A\eta$. Then clearly $s=\#\mathcal{D} \leq \delta k + \#\{d \in \mathcal{D} : d \in (\delta k,k]\}$. Let $p$ be the least prime factor of some $d \in \mathcal{D}$ with $d > \delta k$. Then $d/p$ is a proper divisor of $d$ so by definition $d/p \leq \eta k$, and therefore $p>A$. If $k$ is sufficiently large, then the number of integers $\leq k$ all of whose prime factors are $> A$ is
\begin{align*}
k \prod_{p \leq A}\left(1 - \frac{1}{p}\right) + O_A(1)
\end{align*}
by inclusion-exclusion, and therefore 
\begin{align*}
s &\leq A\eta k + k \prod_{p \leq A}\left(1 - \frac{1}{p}\right) + O_A(1).
\end{align*}
Since $k$ is sufficiently large, the hypothesis of the proposition implies
\begin{align*}
\eta k + s &\leq \eta k + A\eta k + k \prod_{p \leq A}\left(1 - \frac{1}{p}\right) + O_A(1) \\
&= k \left(\eta(A+1) + \prod_{p \leq A}\left(1 - \frac{1}{p}\right) + O_A(k^{-1}) \right) \leq k \left(\frac{c}{2} + O_A(k^{-1}) \right) \leq \frac{2c}{3} k. \qedhere
\end{align*}
\end{proof}

%
\begin{proof}[Proof of Proposition \ref{prop:not many ti are equal to one}]
We recall from Lemma \ref{lem:factorization of n plus i d ell} that $n+id^\ell = a_i t_i^\ell$, where $a_i$ is a positive integer divisible only by primes $<k$ and $t_i$ is divisible only by primes $\geq k$. Furthermore, if $0\leq i < j \leq k-1$, then $\text{gcd}(a_i,a_j) \mid (j-i)$ and $\text{gcd}(t_i,t_j)=1$.

Assume for contradiction that there are $\geq 21$ integers $i \in [0,k-1]$ with $a_i < k$ and $|t_i|=1$. We may then find some such $i_0,j_0$ with $|i_0-j_0|\geq 20$, $1\leq a_{i_0},a_{j_0} < k$, and $|t_{i_0}|=|t_{j_0}|=1$. Therefore
\begin{align*}
20d^\ell \leq |(i_0-j_0)d^\ell| = |n+i_0d^\ell - (n+j_0d^\ell)|\leq |n+i_0d^\ell| + |n+j_0d^\ell| < 2k,
\end{align*}
so $d^\ell \leq k/10$. By the triangle inequality we have
\begin{align*}
|n| = |n+j_0d^\ell -j_0d^\ell| \leq k + kd^\ell \leq k + k^2/10,
\end{align*}
so for any $i \in [0,k-1]$ we have the inequality
\begin{align*}
|n+id^\ell| \leq |n| + kd^\ell \leq k+k^2/5.
\end{align*}

We note that $d\leq (k/10)^{1/\ell} <k/2$, so if $p>k/2$ is a prime we have $p\nmid d$. By Bertrand's postulate we may fix some prime $p \in (k/2,k)$. The integer $n+id^\ell$ is divisible by $p$ if and only if $i\equiv - n\overline{d^\ell} \pmod{p}$, and the number of $i\in[0,k-1]$ such that this congruence holds is either one or two. Observe that $p$ divides $n+id^\ell$ if and only if $p \mid a_i$. Since the product of the $a_i$ is a perfect $\ell$-th power (Lemma \ref{lem:factorization of n plus i d ell}), there is some $a_i$ with $p^2 \mid a_i$. Then
\begin{align*}
k^2/4 < p^2 \leq a_i \leq |n+id^\ell| \leq k + k^2/5,
\end{align*}
and this is a contradiction for $k\geq 21$.
\end{proof}

%
%

\begin{proof}[Proof of Proposition \ref{prop:quant faltings theorem}]
Since $\ell\geq 5$ and the integers $A,B,C$ are nonzero, the curve $\mathcal{C}$ has genus $\geq 2$. Therefore, by Faltings's theorem \cite{Fal1983}, $\mathcal{C}$ has only finitely many rational points. The result then follows from \cite[Th\'eor\`eme 1.1]{Rem2010}.
\end{proof}

\section{The mass increment argument}\label{sec:mass increment}

In this section we prove Proposition \ref{prop:ai < k have pos density}. As described above, the proof relies on an argument that we have termed the \emph{mass increment argument}. The factorization $n+id^\ell = a_it_i^\ell$, described in Lemma \ref{lem:factorization of n plus i d ell}, is crucial. We assume throughout this section that there are nonzero integers $n$ and $t$ and a positive integer $d$ with $\textup{gcd}(n,d)=1$ such that \eqref{eq:prod over AP is ellth power} holds. We also assume familiarity with the properties of the factorizations $n+id^\ell = a_it_i^\ell$ detailed in Lemma \ref{lem:factorization of n plus i d ell}.

The first lemma is essentially due to Erd\H{o}s \cite[Lemma 3]{Erd1955}.

\begin{lemma}\label{lem:subset divides factorial}
Let $k\geq 2$ be an integer. Let $b_0,b_1,\ldots,b_{k-1}$ be positive integers divisible only by primes $<k$ such that whenever $0\leq i < j \leq k-1$ we have $\textup{gcd}(b_i,b_j) \mid (j-i)$. Then there is a subset $S \subseteq \{0,1,\ldots,k-1\}$ with $|S| \geq k-\pi(k)$ such that
\begin{align*}
\prod_{i \in S} b_i \mid (k-1)!.
\end{align*}
\end{lemma}
\begin{proof}
For each prime $p< k$, let $i_p \in [0,k-1]$ be such that $v_p(b_{i_p})$ is maximal. We claim that the set $S = \{0,\ldots,k-1\} \backslash \{i_p : p < k\}$ has the desired property. We clearly have $|S|\geq k-\pi(k)$ by the union bound. Since the $b_i$ are only divisible by primes $<k$, it suffices to show that
\begin{align*}
v_p \Bigg(\prod_{i \in S} b_i \Bigg) \leq v_p((k-1)!)
\end{align*}
for each prime $p<k$.

Fix a prime $p<k$. By the maximality of $v_p(b_{i_p})$ and the hypothesis on the greatest common divisor of pairs,
\begin{align*}
v_p \Bigg(\prod_{i \in S} b_i \Bigg) \leq v_p \Bigg(\prod_{\substack{0\leq i \leq k-1 \\ i \neq i_p}}b_i \Bigg) = v_p \Bigg(\prod_{\substack{0\leq i \leq k-1 \\ i \neq i_p}}\text{gcd}(b_i,b_{i_p}) \Bigg) \leq v_p \Bigg(\prod_{\substack{0\leq i \leq k-1 \\ i \neq i_p}}|i-i_p| \Bigg).
\end{align*}
Splitting into $i\leq i_p-1$ and $i\geq i_p+1$ yields
\begin{align*}
\prod_{\substack{0\leq i \leq k-1 \\ i \neq i_p}}|i-i_p| =i_p! (k-1-i_p)!.
\end{align*}
Since $i_p! (k-1-i_p)!$ is the denominator of the binomial coefficient
\begin{align*}
{{k-1}\choose i_p} = \frac{(k-1)!}{i_p! (k-1-i_p)!}
\end{align*}
and ${{k-1}\choose i_p}$ is a positive integer, it follows that
\begin{align*}
v_p \Bigg(\prod_{\substack{0\leq i \leq k-1 \\ i \neq i_p}}|i-i_p| \Bigg) &\leq v_p((k-1)!). \qedhere
\end{align*}
\end{proof}

\begin{lemma}\label{lem:ai's bigger than k are distinct}
If $n+id^\ell = a_it_i^\ell$ and $a_i \geq k$, then $a_i$ is unique. That is, if $a_i\geq k$, then $a_i \neq a_j$ for all $j \neq i$.
\end{lemma}
\begin{proof}
We proceed by contrapositive. Let $0\leq i < j \leq k-1$ be such that $n+id^\ell = a_it_i^\ell$ and $n+jd^\ell = a_jt_j^\ell$ with $a_i = a_j = \alpha\geq 1$. Then $\alpha(t_j^\ell - t_i^\ell) = (j-i) d^\ell$, so $\alpha \mid (j-i)d^\ell$. By Lemma \ref{lem:factorization of n plus i d ell} we see $\text{gcd}(\alpha,d)=1$ and therefore $\alpha \mid (j-i)$. Since $1\leq j-i \leq k-1$ we have $\alpha \leq k-1$.
\end{proof}

We introduce a bit of notation. Note that, by Lemma \ref{lem:subset divides factorial}, there exists $I \subseteq \{0,1,\ldots,k-1\}$ with $|I|\geq k - \pi(k)$ such that
\begin{align}\label{eq:special subset S dividing factorial}
\prod_{i \in I} a_i \leq k!.
\end{align}
Whenever we write $I$ in this section we shall mean this particular set $I$.

The next lemma is the beginning of the mass increment argument, and provides for the existence of many integers $i$ such that $a_i <k$.

\begin{lemma}\label{lem:spark for mass inc, many i with small ai}
If $k$ is sufficiently large, then there are $\geq 0.3 k/\log k$ integers $i \in I$ such that $a_i < k$.
\end{lemma}
\begin{proof}
Let $R$ denote the number of integers $i \in I$ such that $a_i \geq k$. By Lemma \ref{lem:ai's bigger than k are distinct} all these integers $a_i$ are distinct, so \eqref{eq:special subset S dividing factorial} yields
\begin{align*}
\prod_{j=0}^{R-1} (k+j) \leq \prod_{\substack{i \in I \\ a_i \geq k}} a_i \leq k!.
\end{align*}
We apply Stirling's formula and take logarithms to deduce
\begin{align*}
\sum_{j=0}^{R-1} \log(k+j) \leq k\log k - k + O(\log k).
\end{align*}
The left-hand side is $\geq R \log k$, so $R < k$. By Euler-Maclaurin or Stirling's formula,
\begin{align*}
\sum_{j=0}^{R-1} \log(k+j)&= (R+k)\log(R+k) - k\log k - R + O(\log k),
\end{align*}
and therefore
\begin{align}\label{eq:upper bound on R = num of large ai}
(R+k)\log(R+k) - k\log k - R \leq k \log k - k + O(\log k).
\end{align}
Since $R < k$ we may write $R = (1 - \frac{\eta}{\log k}) k$ for some $\eta > 0$. A rearrangement of \eqref{eq:upper bound on R = num of large ai} and some cancellations give
\begin{align}\label{eq:initial small ai 2 log 2 inequality}
k\Big(2 - \frac{\eta}{\log k} \Big) \log\Big(2 - \frac{\eta}{\log k} \Big) \leq \eta k + O(\log k).
\end{align}
Since $2\log 2 = 1.38\ldots$, we have a contradiction in \eqref{eq:initial small ai 2 log 2 inequality} if $\eta \leq 1.31$, so we must have $\eta \geq 1.31$. Since $R = (1- \frac{\eta}{\log k}) k$ this implies $R \leq (1 - \frac{1.31}{\log k}) k$. Now observe that
\begin{align*}
k - \pi(k) &\leq |I| = \#\{i \in I : a_i < k\} + \#\{i \in I : a_i \geq k\} = \#\{i \in I : a_i < k\} + R \\
&\leq \#\{i \in I : a_i < k\} + \Big(1 - \frac{1.31}{\log k}\Big) k,
\end{align*}
and therefore
\begin{align*}
\#\{i \in I : a_i < k\} \geq 1.31\frac{k}{\log k} - \pi(k) \geq 0.3 \frac{k}{\log k},
\end{align*}
by the prime number theorem and the fact that $k$ is sufficiently large.
\end{proof}

Lemma \ref{lem:spark for mass inc, many i with small ai} is a good start, but it is not immediately suitable. It says there are many elements of $I$ with $a_i < k$, but we do not know how these integers $a_i < k$ are distributed. In particular, we do not know whether there are many distinct $a_i$ or whether there are many $i$ with $a_i$ equal to one another. It is \emph{a priori} possible, for example, that for every such $i$ we have $a_i = \alpha$ for some fixed integer $1\leq \alpha \leq k-1$. Setting aside this problem for the moment, the next lemma shows that if there are many \emph{distinct} integers $a_i < k$, then we obtain an increase in the number of $i \in I$ with $a_i < k$.

\begin{lemma}\label{lem:many distinct ai implies more indices i}
Let $k$ be sufficiently large and $1/100\leq \delta_0 \leq 0.23\log k$. If there are $\geq \delta_0 k/\log k$ distinct integers $a_i <k$ with $i \in I$, then there are $\geq (\delta_0+1/1000)k/\log k$ elements $i \in I$ such that $a_i <k$.
\end{lemma}

\begin{proof}
The proof is similar to that of Lemma \ref{lem:spark for mass inc, many i with small ai}, but we utilize the additional information that we have many distinct integers $a_i < k$ in order to get a smaller upper bound on the number of $i \in I$ with $a_i \geq k$.

Let $R$ denote the number of $i \in I$ with $a_i \geq k$. Since there are $\geq \delta_0 \frac{k}{\log k}$ distinct integers $a_i \in I$ with $a_i < k$, we see by \eqref{eq:special subset S dividing factorial} and Lemma \ref{lem:ai's bigger than k are distinct} that
\begin{align}\label{eq:prod of small ai and big ai divides factorial}
k! \geq \prod_{\substack{i \in I \\ a_i < k}} a_i \cdot \prod_{\substack{i \in I \\ a_i \geq k}} a_i \geq  \Big\lfloor \delta_0 \frac{k}{\log k}\Big\rfloor ! \cdot \prod_{j=0}^{R-1} (k+j).
\end{align}
Stirling's formula implies $\log (\lfloor \delta_0 k/\log k\rfloor !) = \delta_0 k + O \left(\delta_0 k \log \log k/\log k \right)$, so by \eqref{eq:prod of small ai and big ai divides factorial} and another application of Stirling's formula
\begin{align*}
\sum_{i=0}^{R-1} \log(k+j) &\leq k\log k - (1+\delta_0) k + O \left(\delta_0 k \frac{\log \log k}{\log k} \right).
\end{align*}
By the argument leading to \eqref{eq:upper bound on R = num of large ai} in the proof of Lemma \ref{lem:spark for mass inc, many i with small ai}, we deduce
\begin{align}\label{eq:ineq for R with distinctness refinement}
(R+k)\log(R+k) - R \leq 2k\log k - (1+\delta_0)k + O\left(\delta_0 k \frac{\log \log k}{\log k} \right).
\end{align}
Inequality \eqref{eq:ineq for R with distinctness refinement} implies $R < k$, so we write $R = (1 - \frac{\eta}{\log k}) k$ with $\eta > 0$. We insert this expression for $R$ into \eqref{eq:ineq for R with distinctness refinement}, then rearrange and cancel to obtain
\begin{align}\label{eq:distinct ai delta0 ineq}
k\left(2 - \frac{\eta}{\log k} \right) \log \left(2 - \frac{\eta}{\log k} \right) + \delta_0 k  \leq \eta k + O \left(\delta_0 k \frac{\log \log k}{\log k} \right).
\end{align}
Assume that $\eta \leq \delta_0 + 1.005$. Since $\delta_0 \leq 0.23 \log k$, the left-hand side of \eqref{eq:distinct ai delta0 ineq} is $\geq (1-o(1)) k (\delta_0 + 1.77\log(1.77)) \geq k(\delta_0 + 1.01)$. This gives a contradiction, since the right-hand side is $\leq (1+o(1))k (\delta_0 + 1.005)$. Thus, we must have $\eta \geq \delta_0 + 1.005$. It follows that
\begin{align*}
k-\pi(k)\leq |I| &= \#\{i \in I : a_i < k\} + R \leq \#\{i \in I : a_i < k\} + \left(1 - \frac{\delta_0 + 1.005}{\log k}\right) k,
\end{align*}
so
\begin{align*}
\#\{i \in I : a_i < k\} &\geq (\delta_0 + 1.005)\frac{k}{\log k} - \pi(k) \geq \Big(\delta_0 + \frac{1}{1000} \Big) \frac{k}{\log k}. \qedhere
\end{align*}
\end{proof}

The next lemma is the key input: if there are many $i \in I$ with $a_i < k$, then there are many \emph{distinct} integers $a_i < k$. For technical reasons, we work with a slightly stronger notion than just distinctness of the $a_i$. Define a set of integers $\mathcal{A}$ by
\begin{align}\label{eq:defn of mathcal A}
\mathcal{A} = \{1\leq \alpha < k : \text{there is a unique } i \in I \text{ such that }a_i = \alpha\}.
\end{align}
Hence, there are $\geq |\mathcal{A}|$ distinct integers $a_i < k$ with $i \in I$. We wish to show that $|\mathcal{A}|$ is large.

\begin{lemma}\label{lem:ai cant accumulate too much}
Let $1/100\leq \delta_0 \leq 0.23\log k$. Assume there are $\geq (\delta_0 + 1/1000) k/\log k$ indices $i \in I$ such that $a_i < k$. If $\mathcal{A}$ is defined as in \eqref{eq:defn of mathcal A}, then
\begin{align*}
|\mathcal{A}| &\geq \left(\delta_0 + \frac{1}{2000}\right) \frac{k}{\log k}.
\end{align*}
In particular, there are $\geq (\delta_0 + 1/2000) k/\log k$ distinct $a_i < k$ with $i \in I$.
\end{lemma}

In using Lemma \ref{lem:ai cant accumulate too much}, we only need the last statement on the number of distinct $a_i < k$. However, in proving the lemma we need the stronger statement involving the cardinality of $\mathcal{A}$.

We can prove Proposition \ref{prop:ai < k have pos density} assuming Lemma \ref{lem:ai cant accumulate too much}. We restate Proposition \ref{prop:ai < k have pos density} here.

\begin{proposition}\label{prop:restated ai < k have pos density}
Assume that $k$ is sufficiently large and $\ell$ is a prime with $5\leq \ell \leq (\log \log k)^{1/5}$ and $\textup{gcd}(k,\ell)=1$. Assume there are nonzero integers $n$ and $t$ and a positive integer $d$ with $\textup{gcd}(n,d)=1$ such that \eqref{eq:prod over AP is ellth power} holds, and let the integers $a_i$ be as in Lemma \ref{lem:factorization of n plus i d ell}. Then $\#\{a_i : 1 \leq a_i < k\} \geq 0.23 k$.
\end{proposition}
\begin{proof}
Lemma \ref{lem:spark for mass inc, many i with small ai} implies there are $\geq 0.3k/\log k$ elements $i \in I$ such that $a_i < k$. Hence there are $\geq (0.29 + 1/1000) k/\log k$ elements $i \in I$ with $a_i < k$, so Lemma \ref{lem:ai cant accumulate too much} implies there are $\geq 0.29 k/\log k$ distinct integers $a_i < k$ with $i \in I$.

Lemmas \ref{lem:many distinct ai implies more indices i} and \ref{lem:ai cant accumulate too much} together imply that if $1/100\leq\delta_0 \leq 0.23\log k$ and there are $\geq \delta_0 k/\log k$ distinct integers $a_i < k$ with $i \in I$, then there are $\geq (\delta_0 + 1/2000) k/\log k$ distinct integers $a_i < k$ with $i \in I$. In other words, if $\delta_0 \in [1/100,0.23\log k]$, then
\begin{align}\label{eq:mass of ai set increase}
\#\{a_i : i \in I, a_i < k\} \geq \delta_0 \frac{k}{\log k} \Longrightarrow \#\{a_i : i \in I, a_i < k\} \geq \Big(\delta_0 + \frac{1}{2000}\Big) \frac{k}{\log k}.
\end{align}
This allows us to increase the ``mass'' or cardinality of the set $\{a_i : i \in I, a_i < k\}$.

We repeatedly apply \eqref{eq:mass of ai set increase}, starting with $\delta_0 = 0.29$. The size of $\delta_0$ increases by at least $\frac{1}{2000}$ every time we apply \eqref{eq:mass of ai set increase}, so after $\ll \log k$ uses of \eqref{eq:mass of ai set increase} the condition $\delta_0 \in [\frac{1}{100},0.23\log k]$ must fail. If the condition fails then $\delta_0 \geq 0.23\log k$, and therefore
\begin{align*}
\#\{a_i :i \in I, a_i > k\} &\geq 0.23 k. \qedhere
\end{align*}
\end{proof}

We state the following intermediate result which assists in the proof of Lemma \ref{lem:ai cant accumulate too much}.

\begin{lemma}\label{lem:not enough distinct ai gives many points on curve}
Assume $k$ is sufficiently large and $\ell$ is a prime with $5\leq \ell \leq (\log \log k)^{1/5}$ and $\textup{gcd}(k,\ell)=1$. Let $1/100\leq \delta_0 \leq 0.23\log k$. Assume there are $\geq (\delta_0 + 1/1000) k/\log k$ indices $i \in I$ such that $a_i < k$, and that $|\mathcal{A}| < (\delta_0 + 1/2000) k/\log k$, where $\mathcal{A}$ is defined in \eqref{eq:defn of mathcal A}. Then there is some fixed nonzero integer $1\leq |A_0| \leq (\log k)^{O(1)}$ such that
\begin{align*}
t_i^\ell + (-t_j)^\ell = A_0d^\ell
\end{align*}
for $\gg k/(\log k)^{O(1)}$ pairs $(i,j)$. Each such $t_i,t_j$ satisfies $|t_i|,|t_j|\neq 1$.
\end{lemma}

\begin{proof}[Proof of Lemma \ref{lem:ai cant accumulate too much} assuming Lemma \ref{lem:not enough distinct ai gives many points on curve}]
We assume by way of contradiction that $\mathcal{A} < (\delta_0 + 1/2000) k/\log k$. We may then apply Lemma \ref{lem:not enough distinct ai gives many points on curve} to obtain some nonzero integer $|A_0| \ll (\log k)^{O(1)}$ such that
\begin{align*}
t_i^\ell + (-t_j)^\ell = A_0d^\ell
\end{align*}
for $\gg k/(\log k)^{O(1)}$ pairs $(i,j)$, and each such $t_i,t_j$ satisfies $|t_i|,|t_j|\neq 1$. 

We claim that
\begin{align*}
(\pm t_i/d,\pm t_j/d) \neq (\pm t_{i'}/d,\pm t_{j'}/d)
\end{align*}
whenever $(i,j) \neq (i',j')$; here the $\pm$ signs are chosen independently. If $(i,j) \neq (i',j')$, then either $i \neq i'$ or $j \neq j'$. Assume without loss of generality that $i \neq i'$. Since $|t_i|,|t_{i'}|\neq 1$, we see that $|t_i|,|t_{i'}| \geq k$ and are coprime. Then $t_i \neq \pm t_{i'}$, so $\pm t_i/d \neq \pm t_{i'}/d$.

The claim implies that the curve with affine model $X^\ell + Y^\ell = A_0$ contains $\gg k/(\log k)^{O(1)}$ rational points. But this contradicts Proposition \ref{prop:quant faltings theorem} since the upper bound on $\ell$ and $H \ll (\log k)^{O(1)}$ imply
\begin{align*}
\exp(5^{\ell^4} (\log 3H)(\log \log 3H))&\leq \exp((\log k)^{1/2}). \qedhere
\end{align*}
\end{proof}

It therefore suffices to prove Lemma \ref{lem:not enough distinct ai gives many points on curve}. Before doing so, we need a lemma to bound the number of $i \in [0,k-1]$ for which $a_i = \alpha$.
\begin{lemma}\label{lem:at most k by alpha plus 1 indices with ai = alpha}
Let $1\leq \alpha < k$ be an integer and assume that $a_{i_1} = a_{i_2} = \cdots = a_{i_r} = \alpha$ for some distinct integers $i_j \in [0,k-1]$. Then $r \leq k/\alpha + 1$.
\end{lemma}
\begin{proof}
We may assume without loss of generality that $0 \leq i_1 < i_2 < \cdots < i_r \leq k-1$. For any distinct indices $i_j < i_k$, $\alpha(t_{i_k}^\ell - t_{i_j}^\ell)=(i_k-i_j) d^\ell$. Since $\text{gcd}(\alpha,d)=1$ we must have $\alpha \mid (i_k - i_j)$. It follows that $i_{j+1} - i_j \geq \alpha$ for $j\geq 1$, and by induction $i_j \geq (j-1) \alpha$ for $j\geq 1$. The result follows since $(r-1)\alpha \leq i_r \leq k$.
\end{proof}

\begin{proof}[Proof of Lemma \ref{lem:not enough distinct ai gives many points on curve}]
Given an integer $1\leq \alpha < k$, we define $T'(\alpha) = \{i \in I : a_i = \alpha\}$. Obviously $\#T'(\alpha)$ is a nonnegative integer, and $\#T'(\alpha) = 1$ if and only if there is a unique $i \in I$ such that $a_i = \alpha$. It is also obvious that $T'(\alpha) \cap T'(\alpha') = \emptyset$ if $\alpha \neq \alpha'$.

By assumption, there are $\geq (\delta_0 + 1/1000) k/\log k$ elements $i\in I$ such that $a_i < k$. We therefore have
\begin{align}\label{eq:lower bound number of indices with ai < k}
\Big(\delta_0 + \frac{1}{1000}\Big) \frac{k}{\log k} \leq \sum_{\substack{i \in I \\ a_i < k}} 1 = \sum_{1\leq \alpha < k} \#T'(\alpha) = \sum_{\substack{1\leq \alpha < k \\ \#T'(\alpha) \geq 1}} \#T'(\alpha).
\end{align}
We also have the assumption
\begin{align}\label{eq:assuming for contradiction that few alpha are distinct}
|\mathcal{A}| = \sum_{\substack{1 \leq \alpha < k \\ \#T'(\alpha)=1}} 1 < \Big(\delta_0 + \frac{1}{2000}\Big) \frac{k}{\log k}.
\end{align}
Combining \eqref{eq:lower bound number of indices with ai < k} and \eqref{eq:assuming for contradiction that few alpha are distinct} yields
\begin{align}\label{eq:contradictory lower bound many alpha with collisions}
\frac{k}{2000\log k} < \sum_{\substack{1\leq \alpha < k \\ \#T'(\alpha) \geq 2}} \#T'(\alpha).
\end{align}

We now restrict to indices $i$ such that $|t_i| \neq 1$. By Proposition \ref{prop:not many ti are equal to one} there are at most 20 integers $i \in [0,k-1]$ with $a_i < k$ and $|t_i| = 1$. We may assume at least one of these integers $i$ exists, otherwise there is nothing to do. Write $i_1,\ldots,i_r$ for the integers in $[0,k-1]$ such that $a_i < k$ and $|t_{i_j}| = 1$, where $1\leq r \leq 20$. Without loss of generality, we may also assume each $i_j$ is an element of some $T'(\alpha)$ with $\#T'(\alpha) \geq 2$. Let $\alpha_{i_1},\ldots,\alpha_{i_r}$ be such that $i_j \in T'(\alpha_{i_j})$, where it is possible that $\alpha_{i_j} = \alpha_{i_{j'}}$ for $j\neq j'$. There are $\leq 20$ such $\alpha_{i_j}$. The contribution to the right-hand side of \eqref{eq:contradictory lower bound many alpha with collisions} from those $\alpha_{i_j}$ with $\#T'(\alpha_{i_j})\leq 21$ is $\leq 20 \cdot 21 =420$. This implies
\begin{align}\label{eq:many large T alpha with bad i removed}
\frac{k}{3000\log k} < \sum_{\substack{1\leq \alpha < k \\ \#T(\alpha) \geq 2}} \#T(\alpha),
\end{align}
where we define $T(\alpha) = T'(\alpha) \backslash \{i_1,\ldots,i_r\}$. Therefore, if $i \in T(\alpha)$ for some $\alpha$, then $|t_i| \neq 1$.

We next perform a dyadic decomposition on the size of $\alpha$. By \eqref{eq:many large T alpha with bad i removed} and the pigeonhole principle, there exists some $N$ with $1 \ll N \ll k$ such that
\begin{align}\label{eq:after dyadic decompose into N}
\frac{k}{(\log k)^2} \ll \sum_{\substack{N/2 < \alpha \leq N \\ \#T(\alpha) \geq 2}} \#T(\alpha).
\end{align}

We consider two cases, depending on whether or not $N> k/(\log k)^3$. If $N > k/(\log k)^3$, then Lemma \ref{lem:at most k by alpha plus 1 indices with ai = alpha} implies $\#T(\alpha) \ll (\log k)^3$, and we have
\begin{align}\label{eq:when dyadic N is large}
\frac{k}{(\log k)^5} \ll \sum_{\substack{N/2 < \alpha \leq N \\ \#T(\alpha) \geq 2}} 1.
\end{align}

Assume now that $N \leq k/(\log k)^3$. The contribution to \eqref{eq:after dyadic decompose into N} from those $\alpha$ with $\#T(\alpha) \leq \epsilon_0k N^{-1} (\log k)^{-2}$ is $O(\epsilon_0k/(\log k)^2)$, so by taking $\epsilon_0 > 0$ to be a sufficiently small constant we obtain
\begin{align*}
\frac{k}{(\log k)^2} \ll \sum_{\substack{N/2 < \alpha \leq N \\ \#T(\alpha) \gg kN^{-1}(\log k)^{-2}}} \#T(\alpha).
\end{align*}
Lemma \ref{lem:at most k by alpha plus 1 indices with ai = alpha} implies $\#T(\alpha) \ll k/N$, and therefore
\begin{align}\label{eq:there are k over U good alphas}
\frac{N}{(\log k)^2} \ll \sum_{\substack{N/2 < \alpha \leq N \\ \#T(\alpha) \gg kN^{-1}(\log k)^{-2}}} 1.
\end{align}

Fix $\alpha \in (N/2,N]$ with $\#T(\alpha) \gg k N^{-1}(\log k)^{-2}$. We may cover the interval $[0,k-1]$ with $\ll k N^{-1}(\log k)^{-3}$ disjoint intervals of length $\leq N(\log k)^3$. Writing $I_r$ for these intervals, we observe
\begin{align*}
\frac{k}{N(\log k)^2}\ll \#T(\alpha) \leq  \sum_{r \ll k N^{-1}(\log k)^{-3}} |T(\alpha) \cap I_r|.
\end{align*}
Those $r$ with $|T(\alpha) \cap I_r| \leq 1$ contribute a total of $\ll k N^{-1}(\log k)^{-3}$, and therefore
\begin{align*}
\frac{k}{N(\log k)^2} &\ll \sum_{\substack{r \ll k N^{-1}(\log k)^{-3} \\ |T(\alpha) \cap I_r| \geq 2}} |T(\alpha) \cap I_r|.
\end{align*}
Observe now that $|T(\alpha) \cap I_r| \leq 100(\log k)^3$. Indeed, if there are $>100(\log k)^3$ elements of $T(\alpha) \cap I_r$, then by the pigeonhole principle there are distinct $i,j \in T(\alpha)$ with $|i-j| \leq \frac{N}{50}$, but $\alpha(t_i^\ell - t_j^\ell) = (i-j)d^\ell$ implies $\alpha$ divides the nonzero integer $|i-j| \leq \frac{N}{50}$. Since $\alpha > N/2$ this is a contradiction. Therefore
\begin{align}\label{eq:each alpha has many associated rationals}
\frac{k}{N(\log k)^5} &\ll \sum_{\substack{r \ll k N^{-1}(\log k)^{-3} \\ |T(\alpha) \cap I_r| \geq 2}} 1.
\end{align}

We can now produce many pairs $(i,j)$ satisfying the statement of the lemma. Consider first the easier case in which \eqref{eq:when dyadic N is large} holds. Then there are $\gg k/(\log k)^5$ integers $\alpha \gg k/(\log k)^3$ such that $\#T(\alpha) \geq 2$ and $|t_i| \neq 1$ for all $i \in T(\alpha)$. For each such $\alpha$ we arbitrarily choose distinct $i,j \in T(\alpha)$ and then observe that $\alpha(t_i^\ell - t_j^\ell) = (i-j) d^\ell$. Since $\text{gcd}(\alpha,d)=1$ we must have $\alpha \mid (i-j)$, and therefore $t_i^\ell - t_j^\ell = A_{i,j} d^\ell$ for some nonzero integer $A_{i,j}$ satisfying $|A_{i,j}| \ll (\log k)^3$. It follows that there are $\gg k/(\log k)^5$ pairs $(i,j)$ such that $t_i^\ell - t_j^\ell = A_{i,j}d^\ell$ with $1\leq |A_{i,j}| \ll (\log k)^{3}$, so by the pigeonhole principle there is some fixed integer $A_0$ with $1\leq |A_0| \ll (\log k)^3$ such that $t_i^\ell - t_j^\ell = A_0d^\ell$ for $\gg k/(\log k)^{8}$ pairs $(i,j)$.

Now assume $N\leq k/(\log k)^3$, so that \eqref{eq:there are k over U good alphas} holds. There are then $\gg N/(\log k)^2$ integers $\alpha \in (N/2,N]$ such that \eqref{eq:each alpha has many associated rationals} holds. For each integer $r$ such that $|T(\alpha) \cap I_r| \geq 2$, we may choose distinct $i,j$ such that $|i-j| \leq N(\log k)^3$. We then form the ternary equation
\begin{align*}
t_i^\ell - t_j^\ell = \frac{i-j}{\alpha}d^\ell,
\end{align*}
and since $\alpha > N/2$ we see that each $r$ with $|T(\alpha) \cap I_r| \geq 2$ gives rise to a pair $(i,j)$ with
\begin{align*}
t_i^\ell - t_j^\ell = A_{i,j} d^\ell,
\end{align*}
where $A_{i,j}$ is an integer satisfying $1 \leq |A_{i,j}| \ll (\log k)^3$. There are $\gg N/(\log k)^2$ integers $\alpha$ each with $\gg k N^{-1}(\log k)^{-5}$ integers $r$ such that $|T(\alpha) \cap I_r| \geq 2$, so by arguing as before we find there are
\begin{align*}
\gg \frac{N}{(\log k)^2} \cdot \frac{k}{N(\log k)^5} = \frac{k}{(\log k)^7}
\end{align*}
pairs $(i,j)$ such that $t_i^\ell - t_j^\ell = A_{i,j} d^\ell$ for some $1\leq |A_{i,j}| \ll (\log k)^3$. By the pigeonhole principle there exists some $1\leq |A_0| \ll (\log k)^3$ such that there are $\gg k/(\log k)^{10}$ pairs $(i,j)$ with $t_i^\ell - t_j^\ell = A_0 d^\ell$.
\end{proof}

\section{The proof of Theorem \ref{thm:ell = 3}}\label{sec:ell = 3}

We state here some of the necessary tools from the proof of Theorem \ref{thm:ell = 3}. Some of the results could be stated in greater generality, but for the sake of clarity we state and prove results only with the case $\ell=3$ in mind.

The first result gives an upper bound on $|n|$ in terms of $d$ and $k$. It will be used to bound the naive height of points on elliptic curves.

\begin{proposition}\label{prop:ell=3 upper bound on n}
Assume there are nonzero integers $n$ and $t$ and a positive integer $d$ with $\textup{gcd}(n,d)=1$ such that \eqref{eq:prod over AP is ellth power} holds with $\ell=3$. Then $|n|\leq 432k^6d^{18}$.
\end{proposition}

As in the proof of Theorem \ref{thm:main theorem}, we use a mass increment argument to show there are many distinct integers $a_i$.
\begin{proposition}\label{prop:mass inc for ell = 3}
Assume there are nonzero integers $n$ and $t$ and a positive integer $d$ with $\textup{gcd}(n,d)=1$ such that \eqref{eq:prod over AP is ellth power} holds with $\ell=3$. Let the integers $a_i$ be as in Lemma \ref{lem:factorization of n plus i d ell}. Let $\delta>0$ be a sufficiently small constant. If $k$ is sufficiently large and $d\leq \exp(k^{\delta/\log\log k})$, then $\#\{a_i : 1 \leq a_i < k\} \geq 0.23 k$.
\end{proposition}

Lastly, we need a count for rational points on an elliptic curve up to some height.

\begin{proposition}\label{prop:count points in MW lattice}
Let $E/\mathbb{Q}$ be an elliptic curve in quasi-minimal Weierstrass form $E: y^2=x^3+ax+b$, where $a,b \in \mathbb{Z}$. Assume $E$ has rank $r$. If $r\geq 1$, let $L > 0$ be a constant such that $\hat{h}(P) \geq L$ for every nontorsion $P \in E(\mathbb{Q})$, otherwise set $L=1$.  If $H\geq L$, then
\begin{align*}
\#\{P \in E(\mathbb{Q}) : \widehat{h}(P) \leq H\} \leq 16(9H/L)^{r/2}.
\end{align*}
\end{proposition}

\begin{proof}[Proof of Theorem \ref{thm:ell = 3} assuming Propositions \ref{prop:ell=3 upper bound on n}, \ref{prop:mass inc for ell = 3}, and \ref{prop:count points in MW lattice}]
The structure of the proof is similar to the proof of Theorem \ref{thm:main theorem}, only we use Proposition \ref{prop:count points in MW lattice} rather than Proposition \ref{prop:quant faltings theorem}.

Assume there is a nontrivial rational point on the curve $y^3 = x(x+1) \cdots (x+k-1)$ with $k$ sufficiently large and not divisible by 3. By Lemma \ref{lem:rat point on ES implies Dioph eq with ints} we have
\begin{align*}
t^3 = n(n+d^3)\cdots (n+(k-1)d^3)
\end{align*}
with $nt\neq 0$, $d\geq 1$, and $\text{gcd}(n,d)=1$. We may factor $n+id^3 = a_it_i^3$ by Lemma \ref{lem:factorization of n plus i d ell}. If $\delta > 0$ is a sufficiently small positive constant, then we claim that $d\geq \exp(k^{\delta/\log \log k})$.

Assume for contradiction that $d< \exp(k^{\delta/\log \log k})$. By Proposition \ref{prop:mass inc for ell = 3} we have
\begin{align*}
\#\{a_i : 1\leq a_i < k\} \geq 0.23 k.
\end{align*}
There are at most 20 integers $i\in [0,k-1]$ such that $|t_i| = 1$ (Proposition \ref{prop:not many ti are equal to one}), so
\begin{align*}
\#\{a_i : 1\leq a_i < k, |t_i| \neq 1\} \geq 0.23k - 20 \geq 0.229 k
\end{align*}
since $k$ is large. We then apply Proposition \ref{prop:dense sequence has large gcds} and argue as in the proof of Theorem \ref{thm:main theorem} to see that there are fixed, bounded, nonzero integers $A,B,C$ and $\gg k$ pairs $(i,j)$ such that
\begin{align*}
At_i^3 - Bt_j^3 = Cd^3.
\end{align*}
Recalling that $A,B$ are divisors of $a_i,a_j$ obtained after dividing both by $\text{gcd}(a_i,a_j)$, we see that $A$ and $B$ are coprime. This implies $A,B,C$ are pairwise coprime. By folding powers into $t_i,t_j,d$ we may also assume $A,B,C$ are cubefree. Define $\kappa=0$ if $C$ is even and $\kappa=1$ if $C$ is odd. If we set
\begin{align*}
V = 2^{2\kappa}AB \left(2^\kappa A \frac{t_i^3}{d^3} -2^{\kappa-1}C \right), \ \ \ \ \ \ \ \ \ U = 2^{2\kappa}AB \frac{t_it_j}{d^2},
\end{align*}
then $V^2 = U^3 + 2^{6\kappa-2}(ABC)^2$. In the case that $C$ is odd we might have that 4 divides one of $A$ or $B$, in which case $16(ABC)^2$ is divisible by $2^6$. If this occurs we divide through by $2^6$ and adjust $U,V$ accordingly, so that regardless of the parity of $C$ we may consider the elliptic curve
\begin{align*}
E: Y^2 = X^3 + \gamma,
\end{align*}
where $\gamma$ is a bounded, nonzero integer which is not divisible by the sixth power of any prime. The rank $r$ of $E$ is bounded since $|\gamma|$ is bounded. Note that the $j$-invariant of $E$ is zero. We also note that the mapping of triples $(t_i,t_j,d)$ satisfying $At_3^3-Bt_j^3 = Cd^3$ to rational points $(X,Y)$ satisfying $Y^2 = X^3+\gamma$ is injective.

From Proposition \ref{prop:ell=3 upper bound on n} and the triangle inequality we find
\begin{align*}
|t_i|^3 \leq a_i|t_i|^3 = |n+id^3| \leq 432k^6d^{18} + kd^3 \leq 433k^6d^{18},
\end{align*}
so $|t_i| \leq (kd)^{O(1)}$. Since there are $\gg k$ pairs $(i,j)$ such that $At_i^3 - Bt_j^3 = Cd^3$, we see that the elliptic curve $E$ has $\gg k$ rational points $P$ with naive height $ H(P)\ll (kd)^{O(1)}$. Therefore, $h(P) \ll \log(kd) \ll (\log k)(\log 2d)$. Since $|\gamma|$ is bounded, work of Silverman \cite{Sil1990} implies the difference between $h(P)$ and the canonical height $\hat{h}(P)$ is bounded, and we obtain
\begin{align*}
k \ll \#\{P \in E(\mathbb{Q}) : \widehat{h}(P) \leq H\},
\end{align*}
where $H = C(\log k)(\log 2d)$ for some positive constant $C$. Since $| \gamma|$ is bounded there is some constant $L>0$ such that $\widehat{h}(P) \geq L$ for every $P \in E(\mathbb{Q})/E_{\text{tors}}(\mathbb{Q})$, and $H\geq L$ since $k$ is sufficiently large. By Proposition \ref{prop:count points in MW lattice} and the fact that $r$ is bounded we obtain
\begin{align}\label{eq:small d gives contradiction for ell=3}
k \ll H^{O(1)} \ll (\log k)^{O(1)} (\log d)^{O(1)}.
\end{align}
By assumption we have $\log d \leq k^{\delta/\log \log k}$, but then \eqref{eq:small d gives contradiction for ell=3} is a contradiction since $\delta>0$ is a constant and $k$ is sufficiently large.
\end{proof}

\begin{proof}[Proof of Proposition \ref{prop:mass inc for ell = 3}]
We proceed by a mass increment argument, referring the reader back to Section \ref{sec:mass increment} for the majority of the details. The elliptic curve arguments are similar to the proof of Theorem \ref{thm:ell = 3} above, but we need to be slightly more careful with issues of uniformity since the elliptic curve no longer has bounded coefficients. 

The crucial point in the mass increment argument is obtaining a contradiction from the analogue of Lemma \ref{lem:not enough distinct ai gives many points on curve}. In the present situation, this entails deducing a contradiction from the following hypotheses:
\begin{itemize}
\item $d\leq \exp(k^{\delta/\log \log k})$
\item There exists an integer $1\leq |A_0| \ll (\log k)^{O(1)}$ such that there are $\gg k/(\log k)^{O(1)}$ pairs $i,j$ with
\begin{align}\label{eq:diff of cubes}
t_i^3 - t_j^3 = A_0d^3
\end{align}
and $|t_i|,|t_j| \neq 1$.
\end{itemize}

By folding factors of $A_0$ into $d$ we may assume $A_0$ is cubefree. We then change variables
\begin{align*}
U &= 12A_0 \frac{d}{t_i-t_j}, \ \ \ \ \ \ \ V = 36A_0 \frac{t_i+t_j}{t_i-t_j}
\end{align*}
and obtain $V^2=U^3-432A_0^2$. The integer $432 A_0^2$ may be divisible by $2^6$ or $3^6$ (or both), but by dividing through as necessary and adjusting $U,V$ we arrive at the elliptic curve
\begin{align*}
E : Y^2 = X^3 - D,
\end{align*}
where $D \ll (\log k)^{O(1)}$ is a positive integer which is not divisible by the sixth power of any prime. Since the point
\begin{align*}
\left(12A_0 \frac{d}{t_i-t_j}, 36A_0 \frac{t_i+t_j}{t_i-t_j}\right)
\end{align*}
coincides with another such point (replacing $t_i$ by $t_i'$ and $t_j$ by $t_j'$) only when $t_i'=t_i, t_j' = t_j$, we see by Proposition \ref{prop:ell=3 upper bound on n} that $E$ has $\gg k/(\log k)^{O(1)}$ rational points with naive height $\leq (kd)^{O(1)}$. Since $E$ has $j$-invariant zero and discriminant $|\Delta| \ll (\log k)^{O(1)}$, the difference between the Weil height $h(P)$ and the canonical height $\hat{h}(P)$ satisfies $|h(P) - \hat{h}(P)| \ll \log \log k$ by \cite[Theorem 1.1]{Sil1990}. We deduce
\begin{align}\label{eq:lower bound on point count for ell curve}
\frac{k}{(\log k)^{O(1)}} \ll \#\{P \in E(\mathbb{Q}) : \widehat{h}(P) \leq H\},
\end{align}
where $H = C(\log k)(\log 2d)$ for some absolute constant $C>0$. Since $E$ has integral $j$-invariant, there is an absolute constant $L >0$ such that $\hat{h}(P) \geq L$ uniformly in $D$ and nontorsion $P$ by \cite{Sil1981}.

By a descent argument \cite[Proposition 2]{Fou1993} (restated in \cite[Lemma 4.1]{HelVen2006}) the rank $r$ of $E$ satisfies $r \ll 1 + \omega(D) + \log h_3(\mathbb{Q}(\sqrt{-D}))$, where $\omega(D) \ll \log \log k$ is the number of distinct primes dividing $D$ and $h_3(\mathbb{Q}(\sqrt{-D})$ is the size of the 3-torsion in the class group of $\mathbb{Q}(\sqrt{-D})$. We use the trivial bound $h_3(\mathbb{Q}(\sqrt{-D}) \leq h(-D)$, where $h(-D)$ is the class number of $\mathbb{Q}(\sqrt{-D})$. From the Dirichlet class number formula \cite[Chapter 6]{Dav2000} and the trivial bound $L(1,\chi_{-D}) \ll \log 3D$ \cite[p. 96]{Dav2000} we deduce that $r \ll \log \log k$.

We apply Proposition \ref{prop:count points in MW lattice} and \eqref{eq:lower bound on point count for ell curve} to obtain
\begin{align*}
\frac{k}{(\log k)^{O(1)}} \ll (\log k)^{O(r)} (\log d)^{O(r)}.
\end{align*}
As $r \ll \log \log k$ and $k$ is sufficiently large this implies $k^{1/2} \ll (\log d)^{O(\log \log k)}$, say. However, we have assumed the upper bound $\log d \leq k^{\delta/\log \log k}$ with $\delta>0$ sufficiently small, and this is a contradiction.
\end{proof}

\begin{proof}[Proof of Proposition \ref{prop:count points in MW lattice}]
Write $E_{\text{tors}}(\mathbb{Q})$ for the subgroup of torsion points. By work of Mazur \cite{Maz1977,Maz1978} we have $\#E_{\text{tors}}(\mathbb{Q}) \leq 16$, so the proposition is true if $r=0$. Hence, we may assume $r\geq 1$.

We reduce to counting nontorsion points. Every $P \in E(\mathbb{Q})$ can be written uniquely as $P = T + Q$, where $T$ is a torsion point and $Q \in E(\mathbb{Q})/E_{\text{tors}}(\mathbb{Q})$ is a nontorsion point. By Mazur's bound we have $\#\{P \in E(\mathbb{Q}) : \| P \| \leq H\} \leq 16\#\{Q \in E(\mathbb{Q})/E_{\text{tors}}(\mathbb{Q}) : \| T+Q \| \leq H\}$ for some torsion point $T$. As $T$ is torsion we have $\|T+Q\| = \|Q\|$. It therefore suffices to prove $\#\{P \in E(\mathbb{Q})/E_{\text{tors}}(\mathbb{Q}) : \widehat{h}(P) \leq H\} \leq (9H/L)^{r/2}$.

By work of Naccarato (see \cite[p.502--505]{Nac2021}, also \cite{BZ2004,Duj2023preprint}), we have
\begin{align}\label{eq:nac bound}
\#\{P \in E(\mathbb{Q})/E_{\text{tors}}(\mathbb{Q}) : \widehat{h}(P) \leq H\} &\leq \left(1 + 2\sqrt{H/L} \right)^r.
\end{align}
We obtain the result by noting $1+2\sqrt{H/L} \leq 3\sqrt{H/L}$, since we assume $H \geq L$.
\end{proof}

It remains to prove Proposition \ref{prop:ell=3 upper bound on n}. We do so using the following lemma, which is a weak form of a result due to Erd\H{o}s (see \cite{Erd1968}). 

\begin{lemma}\label{lem:prods all distinct}
Let $\delta > 0$ be a positive constant, and let $x$ be sufficiently large in terms of $\delta$. Let $1 \leq m_1 < m_2 < \cdots < m_T \leq x$ be distinct positive integers such that $m_i m_j \neq m_r m_s$ whenever $i \neq r, i \neq s$. Then $T \leq \delta x$.
\end{lemma}

\begin{proof}[Proof of Proposition \ref{prop:ell=3 upper bound on n} assuming Lemma \ref{lem:prods all distinct}]

The proof is somewhat inspired by the proof of \cite[Lemma 1]{ErdSel1975}.

We make two claims.

First, we claim that if there exist $i\neq j$ with $a_i = a_j$, then $|n| \leq 2k^{3/2}d^{9/2}$. If such $i\neq j$ exist, then $kd^3 \geq |i-j|d^3 = |a_it_i^3 - a_jt_j^3|$. We must have $t_i\neq t_j$ since $i\neq j$. If $t_i = -t_j$ then $|a_it_i^3 - a_jt_j^3| = 2a_j|t_j|^3$ since $a_i=a_j$. If $|t_i| \neq |t_j|$, then without loss of generality $|t_i| > |t_j|$ and
\begin{align*}
|a_it_i^3 - a_jt_j^3| \geq a_j (|t_i|^3-|t_j|^3) \geq a_j ((|t_j|+1)^3 - |t_j|^3) \geq a_j t_j^2.
\end{align*}
Therefore, in any case
\begin{align*}
kd^3 \geq a_jt_j^2 \geq (a_j|t_j|^3)^{2/3} = |n+jd^3|^{2/3},
\end{align*}
so $|n+jd^3| \leq k^{3/2}d^{9/2}$. By the triangle inequality we obtain
\begin{align*}
|n| = |n+jd^3 - jd^3| \leq k^{3/2}d^{9/2} + kd^3 \leq 2k^{3/2}d^{9/2},
\end{align*}
and this proves the claim.

Second, we claim that if $a_ia_j=a_ra_s$ for some indices $i\neq r, i \neq s$, then $|n| \leq 432k^6 d^{18}$. Since $a_i \mid a_r a_s$, we have
\begin{align*}
a_i = \text{gcd}(a_i,a_ra_s) \leq \text{gcd}(a_i,a_r) \cdot \text{gcd}(a_i,a_s) < k^2,
\end{align*}
the last inequality following from Lemma \ref{lem:factorization of n plus i d ell} and the fact that $i\neq r,i \neq s$. If $(n+id^3)(n+jd^3)=(n+rd^3)(n+sd^3)$ then $a_ia_jt_i^3t_j^3 = a_ra_s t_r^3t_s^3$, and therefore $t_it_j = t_rt_s$. Then $t_i \mid t_rt_s$, but $\text{gcd}(t_i,t_rt_s)=1$ by Lemma \ref{lem:factorization of n plus i d ell} again, so $t_i = \pm 1$. Then $|n+id^3| = a_i < k^2$, so
\begin{align*}
|n| \leq |n+id^3| + id^3 \leq k^2 + kd^3 \leq 2k^2d^3.
\end{align*}
Now assume $(n+id^3)(n+jd^3)\neq (n+rd^3)(n+s d^3)$. We deduce
\begin{align*}
d^6(ij-rs) + nd^3 (i+j-r-s) &= (n+id^3)(n+jd^3)- (n+rd^3)(n+s d^3) \\ 
&= a_ia_j ((t_it_j)^3-(t_rt_s)^3).
\end{align*}
The left-hand side has absolute value $\leq k^2d^6 + 2k d^3|n|\leq 3k^2d^6|n|$. Since $t_it_j \neq t_rt_s$, we see that if $|t_rt_s| = |t_it_j|$, then the right-hand side has absolute value 
\begin{align*}
\geq 2a_ia_j |t_it_j|^3 \geq |n+id^3| \cdot |n+jd^3| \geq (|n|-kd^3)^2.
\end{align*}
Otherwise, we may assume without loss of generality that $|t_rt_s| > |t_it_j|$, and then the right-hand side is
\begin{align*}
&\geq a_ia_j ((|t_it_j|+1)^3-|t_it_j|^3) \geq a_ia_j |t_it_j|^2 \geq (a_i |t_i|^3 a_j |t_j|^3)^{2/3} \\
&= (|n+id^3| \cdot |n+jd^3|)^{2/3} \geq (|n| - kd^3)^{4/3}.
\end{align*}
In any case, the right-hand side is therefore $\geq (|n|-kd^3)^{4/3}$. If $|n| \leq 2kd^3$ then we already have a suitable upper bound on $|n|$, and if $|n| \geq 2kd^3$ then $|n|-kd^3 \geq |n|/2$, and we find $|n|/2 \leq (3k^2d^6|n|)^{3/4}$. This implies $|n| \leq 432 k^6d^{18}$, which proves the second claim.

By Lemma \ref{lem:subset divides factorial}, there exists a set $S \subseteq \{0,1,\ldots,k-1\}$ with $|S| \geq k-\pi(k)$ such that
\begin{align*}
\prod_{i \in S} a_i \leq k!.
\end{align*}
Now assume for contradiction that $|n| > 432 k^6 d^{18}$. By the contrapositive of the two claims above, this implies that all the $a_i$ are distinct from one another, and that $a_ia_j \neq a_ra_s$ whenever $i \neq r,i \neq s$.  If we relabel the $a_i$ so that $1\leq a_1 < a_2 < \cdots$, then Lemma \ref{lem:prods all distinct} with $\delta = 1/10$, say, implies $a_i > 3i$ for all $i$ sufficiently large.

 We therefore have
\begin{align*}
k!\geq \prod_{i \in S} a_i \geq \prod_{i=i_0}^{k-\pi(k)} (3i)
\end{align*}
for some large constant $i_0$. By Stirling's formula and the prime number theorem we find $k! \geq  3^k e^{-(1+o(1))k} k!$, and this is a contradiction for $k$ sufficiently large.
\end{proof}

Erd\H{o}s's proof of (his version of) Lemma \ref{lem:prods all distinct} is moderately involved and relies on arguments from graph theory. We give an easier proof using Proposition \ref{prop:dense sequence has large gcds}.

\begin{proof}[Proof of Lemma \ref{lem:prods all distinct}]
We proceed by contrapositive, so assume $T > \delta x$. We apply Proposition \ref{prop:dense sequence has large gcds} with $c = \delta, A = \exp((\varepsilon \delta)^{-1})$, and $\eta = A^{-2}$, where $\varepsilon>0$ is a sufficiently small constant. Then there are $> \frac{\delta}{3} x$ ordered pairs $(m_i,m_j)$ of distinct integers with $\text{gcd}(m_i,m_j) > \eta x$.

We first show any integer $m$ appears in $\leq \eta^{-3}$ of the ordered pairs. After relabeling, we may assume that $\text{gcd}(m,m_j) > \eta x$ for $1\leq j \leq J$ and distinct integers $m_j$ with $m_j \neq m$. We wish to prove $J\leq \eta^{-3}$, so assume for contradiction that $J > \eta^{-3}$. The integer $m$ has $\leq \eta^{-1}$ divisors $> \eta x$, so by the pigeonhole principle there are $> \eta^{-2}$ choices of $j$ such that $\text{gcd}(m,m_j) = d > \eta x$, for some fixed divisor $d$ of $m$. We apply the pigeonhole principle to this set of $j$'s to find distinct $j$ and $j'$ with $|m_j - m_{j'}| \ll \eta^2 x$. Since $\eta x < d \leq |m_j - m_{j'}| \ll \eta^2 x$ and $\eta$ is small, we obtain a contradiction. It follows that $J \leq \eta^{-3}$.

We now iteratively build a set $S$ of ordered pairs such that each integer appears in at most one element of $S$. First, choose some ordered pair $o_1$ of integers as above and add it to $S$. Remove from consideration the $\leq 2\eta^{-3}$ other ordered pairs which could have an integer in common with $o_1$. If $S = \{o_1,\ldots,o_n\}$, then there are $> \frac{\delta}{3}x - 2n\eta^{-3}-n$ ordered pairs from which we may still choose. It follows that we may construct $S$ such that $|S| > \lambda x$, where $\lambda>0$ is a sufficiently small constant depending on $\eta$.

Consider now all the ordered pairs in $S$. By relabeling, we may write $\text{gcd}(m_i,m_{i+1}) = d_{i}> \eta x$, where $1\leq i \leq |S|/2$. Furthermore, the integers $m_i$ appearing in the ordered pairs are all distinct. Since $m_i / d_i,m_{i+1}/d_i < \eta^{-1}$ and $x$ is sufficiently large, the pigeonhole principle yields the existence of $i \neq k$ with $\left(m_i/d_i, m_{i+1}/d_i \right) = \left(m_k/d_k, m_{k+1}/d_k \right)$. Thus $m_i/d_i = m_k / d_k$ and $m_{i+1}/d_i = m_{k+1}/d_k$, and it follows that
\begin{align*}
\frac{m_i}{d_i} \frac{m_{k+1}}{d_k} = \frac{m_k}{d_k}\frac{m_{i+1}}{d_i},
\end{align*}
so $m_i m_{k+1} = m_k m_{i+1}$. Since $i\neq k, i \neq i+1$, we obtain the desired result.
\end{proof}

\section*{Acknowledgements}

The author began this work while supported by a post-doctoral research fellowship at All Souls College, University of Oxford. The author is supported by National Science Foundation grant DMS-2418328.

\bibliographystyle{plain}
\bibliography{refs}

\end{document}